\documentclass[11pt]{amsart}

\usepackage[applemac]{inputenc}
\usepackage{times, ulem, amsfonts}
\usepackage{mathpazo, amsmath, amssymb, amsthm,color}
\usepackage[colorlinks=true, pdfstartview=FitV, linkcolor=blue, citecolor=blue, urlcolor=blue, dvips, breaklinks]{hyperref}

\newcommand{\BMO}{{{\rm BMO}}}

\newcommand{\osc}{{\rm Osc}}

\definecolor{sah}{rgb}{0.66,0.33, 0.04}
\definecolor{adel4}{cmyk}{1,0,0,0}
\definecolor{adel3}{rgb}{0.66,0.33, 0.04}
\definecolor{adel1}{cmyk}{0,0.20,1,0}
\definecolor{adel2}{cmyk}{0,0.40,1,0.30}
\definecolor{adel0}{rgb}{0.99,0.60, 0.30}
\definecolor{trut}{rgb}{0.99,0.80, 0.00}
\definecolor{trus}{rgb}{0.00, 0.50, 0.00}
 \definecolor{trust}{rgb}{0.99, 0.99, 0.80}
\definecolor{MaCouleur}{rgb}{0,0.9,0.3}
\def\virgp{\raise 2pt\hbox{,}}
\def\({\left(}
\def\){\right)}

\def\<{\langle}
\def\>{\rangle}

\theoremstyle{plain}

\numberwithin{equation}{section}
\newtheorem{theorem}{Theorem}[section]
\newtheorem{proposition}[theorem]{Proposition}
\newtheorem{lemma}[theorem]{Lemma}
\newtheorem{corollary}[theorem]{Corollary}
\newtheorem{definition}[theorem]{Definition}

\newtheorem{Conj}[theorem]{Conjecture}
\theoremstyle{remark}
\newtheorem{remark}[theorem]{Remark}
\newcommand{\av}{{\rm Avg}}

\newcommand{\R}{{\mathbb R}}

  \newcommand{\Lip}{\textnormal{Lip}}

\newcommand{\lb}{{{\rm LBMO}}}

\usepackage[textmath,displaymath,floats,graphics,delayed]{preview}
\newcommand{\LMO}{{\it L}^{\alpha}{\it mo}}
 
 \newcommand{\LMOI}{ {\it Lmo}}

  \title[2D Inviscid Limit for Euler equations]{On the inviscid limit of the 2D Euler equations with vorticity along the $(\LMO)_\alpha$ scale}

\author[F. Bernicot]{Fr\'ed\'eric Bernicot}
\address{CNRS - Universit\'e de Nantes \\ Laboratoire de Math\'ematiques Jean Leray \\ 2, Rue de la Houssini\`ere 44322 Nantes Cedex 03, France}
 \email{frederic.bernicot@univ-nantes.fr}

\author[T. Elgindi]{Tarek Elgindi}
\address{Courant Institute of Mathematical Sciences \\ 251 Mercer Street \\ New York 10012-1185 NY \\ USA}
 \email{elgindi@cims.nyu.edu}

\author[S. Keraani]{Sahbi Keraani}
\address{UFR de math\'ematiques \\ Universit\'e de Lille 1\\
 59655 Villeneuve d'Ascq  Cedex\\   France}
\email{sahbi.keraani@univ-lille1.fr}

\thanks{This work was partially supported by the ANR under the project AFoMEN no. 2011-JS01-001-01.}

\date{\today}


\begin{document}

\begin{abstract} In a recent paper  \cite{ BK2}, 
the global well-posedness of the two-dimensional Euler equation with vorticity in \mbox{$L^1\cap \lb$} was proved, where $ \lb$ is a  Banach space which is strictly imbricated between  \mbox{$L^\infty$}  and  $\BMO$. In
the present paper we  prove a global result of inviscid limit of the Navier-stokes system with data in this space and other spaces with the same BMO flavor. Some results of local uniform estimates on solutions of the Navier-Stokes equations, independent of the viscosity, are also obtained.
\end{abstract}

\maketitle

\begin{quote}
\footnotesize\tableofcontents
\end{quote}
\section{Introduction}
In this work, we consider the problem of the inviscid  limit of the 2D-Navier Stokes equations with rough initial data.
More precisely, we are interested in the situation where the vorticity lives in specific Morrey-Campanato spaces (in the same flavor as already studied in \cite{BK2,BH} and very recently in \cite{CMZ}).
Morrey-Campanato spaces are Banach spaces which extend the notion of $\BMO$ (the space of functions with bounded mean oscillation) describing situations where the oscillation of the function in a ball is controlled with respect to the radius of the ball. These spaces have attracted much attention in the last few decades due to some specific properties of them (John-Nirenberg inequalities, Duality with Hardy spaces, etc.). For example, the theory of Morrey-Campanato spaces may come in useful when the Sobolev embedding theorem is not available and have proven to be very useful in the study of elliptic PDEs.

We do not detail the literature about these spaces since it is huge. In this current work, we only focus on the $\LMO$ spaces (see precise definitions in Section \ref{sec:def}) where the oscillations of a function on a ball of radius $r\ll 1$ are bounded by $|\log(r)|^{-\alpha}$. What is interesting, is that the scale $(\LMO)_{0<\alpha<1}$ can be thought as an intermediate scale between $\BMO$ (for $\alpha \to 0$) and $L^\infty$ (for $\alpha \to 1$).

The Navier-Stokes system is the basic mathematical model for viscous incompressible flows and reads as follows:
 \begin{equation}
 \label{NS}
   (NS_\epsilon) \qquad \left\{ 
 \begin{array}{ll} 
 \partial_t u^\varepsilon +  u^\varepsilon \cdot\nabla u^\varepsilon-\varepsilon\Delta u^\varepsilon+\nabla P^\varepsilon =0, \\
 \nabla.u^\varepsilon=0,  \\
u^\varepsilon_{\mid t=0}= u_0.
 \end{array} \right.    
      \end{equation}
Associated to the viscosity parameter $\epsilon$, the vector field $u^\varepsilon$ stands for the velocity of the fluid, the quantity $P^\varepsilon$ denotes the scalar pressure, and $\nabla.u^\varepsilon=0$ means that the fluid is incompressible. We also detail the fractional Navier-Stokes equation, of order $\alpha\in(0,1)$:
 \begin{equation}
 \label{fracNS}
    \left\{ 
 \begin{array}{ll} 
 \partial_t u^\varepsilon +  u^\varepsilon \cdot\nabla u^\varepsilon+\varepsilon(-\Delta)^{\frac{\alpha}{2}} u^\varepsilon+\nabla P^\varepsilon =0, \\
 \nabla.u^\varepsilon=0,  \\
u^\varepsilon_{\mid t=0}= u_0.
 \end{array} \right.    
      \end{equation}
where the diffusion term is given by the fractional power of the Laplacian operator. When we neglect the diffusion term, then we obtain the Euler equations,
\begin{equation}
 \label{E}
  (E) \qquad \left\{ 
 \begin{array}{ll} 
 \partial_t u +  u \cdot\nabla u+\nabla P =0, \\
 \nabla.u=0,\\
 u_{\mid t=0}= u_0.
 \end{array} \right.   
      \end{equation}
      The mathematical study of the Navier-Stokes system was initiated by
  Leray in his pioneering \mbox{work \cite{Leray}.} In fact, by using a compactness
 method, he  proved
that for any divergence-free initial \mbox{data $v^0$} in the energy space
 $L^2,$ there exits  a global solution to $(NS^{\varepsilon})$. In the  case of
 {\it two} dimensional space that weak solution was proven to be unique.
 However, for higher 
dimension ($d\geq 3$) the problem of uniqueness is still a widely open
 problem. In the 60's,  \mbox{Fujita-Kato \cite{fk}} exhibited for initial data
 lying in the critical  Sobolev space $\dot{H}^{\frac{d}{2}-1}$ a class of 
 unique local solutions called mild solutions. We emphasize that
  the same result holds true when the initial data belong to the inhomogeneous Sobolev space $H^s,$ with $s\geq \frac{d}{2}-1.$
The global existence of these solutions  is  an outstanding open problem. However a positive answer is given at least in both following cases: 
  either when the initial data is small in the critical space $\dot{H}^{\frac{d}{2}-1}$ which is invariant under the scaling of the the Navier-Stokes equations,
   or in the space dimension two (this is because in two dimensions the scale invariant space is energy space). 
   
In the two dimensional space and when the regularity is sufficient to give a sense to the Biot-Savart law, then one can consider an alternative weak formulation:  the vorticity-stream weak formulation. It  consists in resolving the weak form of \eqref{E} in terms of vorticity $\omega=\textrm{curl}(u)$:  
     \begin{equation}
 \label{tourbillon}
 \partial_t\omega+(u \cdot \nabla)\omega=0,
 \end{equation}
    supplemented with the Biot-Savart law:
\begin{equation*}
u=K\ast\omega,\quad \hbox{with}\quad K(x)=\frac{x^\perp}{2\pi|x|^2}.
\end{equation*}
   The questions of existence/uniqueness of weak solutions have been extensively studied (see \cite{Ch1, BM, lions1} for instance). We emphasize that, unlike the fixed-point argument, the compactness method does not guarantee the uniqueness of the  solutions and then  the two issues (existence/uniqueness) are usually  dealt with separately.  These questions have been originally addressed by Yudovich in \cite{Y1} in the context of the Euler equations where the existence and uniqueness of weak solution to 2D Euler systems (in a bounded domain) are proved under the assumptions:   \mbox{$u_0\in L^2$}  and  \mbox{$\omega_0\in L^\infty$}.  
   Many works have been dedicated to the extension of this result to more general spaces (see \cite{Ser,DM,GMO,Ch,tan,Vishik1,Y2,De,FLX,Ger} for instance).  To the best of our knowledge all these contributions lack the proof of  at least  one of the following three fundamental properties:  global existence, uniqueness and regularity persistence. In \cite{BK2} we have extended Yudovich's result to some class of initial vorticity in a Banach space which is strictly imbricated between  \mbox{$L^\infty$}  and  \mbox{${\rm BMO}$} for which one has the following three fundamental properties:  global existence, uniqueness and regularity persistence.

      The problem of the convergence of smooth viscous solutions of \eqref{NS}   to the
Eulerian one as $\varepsilon$ goes to zero is well understood (in the case of the whole space of the torus). Majda showed
that under the assumption $v^0 \in  H^s$	with $s > d + 2$, the solutions $(u^\varepsilon)_{\varepsilon>0} $
converge in $L^2$ norm when $\varepsilon$ goes to zero to the unique solution  of \eqref{E}. The convergence rate is of order $(\varepsilon t)^{\frac{1}{2}}$. 
This result has been improved by Masmoudi \cite{Mas}.
For Yudovich type solutions with only the assumption that the vorticity is bounded this question was resolved by Chemin \cite{Ch3}.

 The first result of this paper is the following (in Section \ref{sec:def} we recall for the definitions of the spaces).

 \begin{theorem}
\label{il1}Assume $p\in [1,2)$. Let  $u_0\in L^2(\mathbb R^2)$ a divergence free vector fields such that $\omega_0\in L^p \cap \lb$ and  $u_\varepsilon$ {\rm (} resp. $u$ {\rm )} the solution of $(NS_\epsilon)$ {\rm (resp. (E))}. Then, for every $T>0$ there exist $C=C(u_0)$ and $\varepsilon_0=\varepsilon_0(u_0,T)$ such that
$$ \|u^\varepsilon(t)-u(t)\|_{L^2(\mathbb R^2)}\leq (Ct\varepsilon)^{\frac{1}{2}\exp(1-e^{Ct})},\qquad \forall t\in [0,T] ,\quad\forall \varepsilon\leq \varepsilon_0.$$
\end{theorem}

\begin{remark} In \cite{BK2} the global existence and uniqueness for 2D Euler with  initial vorticity $\omega_0\in L^p \cap \lb$ has been proved. The additional assumption $u_0\in L^2(\mathbb R^2)$ is easily propagated and we get $u\in  L^\infty  L^2$. 
\end{remark}

The second result is the counterpart version for more regular initial data, with an improved rate of convergence:

\begin{theorem}
\label{il1-bis} Assume $p\in [1,2)$. Let  $u_0\in L^2(\mathbb R^2)$ a divergence free vector field such that $\omega_0=\textrm{curl}(u_0) \in L^p \cap \LMOI$. 
Then there exists a unique solution of the 2D incompressible Euler equations (\ref{E}) such that for every $\delta>0$, $\omega \in L^\infty _{\text{loc}}( [0,\infty);L^{\alpha}mo \cap L^p),$ where $$\alpha(t)=1-\sqrt{t} , \, \, 0\leq t \leq \delta^2,$$
$$=1-\delta, \, \, t>\delta^2.$$ 
Moreover for every $\delta\in(0,1)$ there exist $C=C(u_0,\delta)$ and $\varepsilon_0=\varepsilon_0(u_0,T,\delta)$ such that
$$
\|u^\varepsilon(t)-u(t)\|_{L^2(\mathbb R^2)}\leq (CT\varepsilon)^{\frac{1}{2}e^{\beta(t)}},\qquad \forall t\in [0,T] ,\quad\forall \varepsilon\leq \varepsilon_0,
$$
with
$$ \beta(t) = \max(1-\delta,(1-\frac{e^{C_0t}-1}{2})^{1\over \delta}).$$
\end{theorem}
\begin{remark} The first part gives a global existence of solution for Euler equations, with a loss of regularity as small as we want (since $1-\delta \leq \alpha\leq 1$ and $\delta$ is arbitrary small). This improves some results of \cite{CMZ} in the particular situation of $\LMO$ with $\alpha=1$.
\end{remark}

\begin{remark} The order rates of convergence (of Theorems \ref{il1} and \ref{il1-bis}) are equal to $\frac{1}{2}$ at $t=0$ and then they are decreasing with the time. Moreover, the order of rate of convergence in Theorem 1.2 is bigger than $\frac{1-\delta}{2}$ (for $\delta$ as small as we want) which is just below the optimal rate in the case of $\frac{1}{2}$ strong solutions. This rate beats all of the previous rates of convergence for weak solutions: for example, the rate given in the case of weak solutions with bounded vorticity is exponentially decaying in time \cite{Ch3}. See also \cite{CozzK} and \cite{K}. 
\end{remark}

\begin{remark}  Since the $L^\infty$-norm of  $(u^\varepsilon)_{\varepsilon>0}$ is uniformly  bounded then, by interpolation, 
the convergence to the Eulerian solution $u$ holds in every $L^q$ with  $q\in [2,+\infty[$.
\end{remark}

 The uniform (with respect to the viscosity parameter $\epsilon$) bound of the  family of solutions to \eqref{NS} in the adequate space remains essentially open. The difficulty is due to the nature this norm which prevents us from dealing with a transport and advection at the same time.  To overcome this difficulty we use an idea which is based on Trotter's formula: we discretize the time and  alternate  the Euler and Heat equations in the small intervals and then let   length of the interval  goes to $0$.
The implementation of this algorithm is heavily related to the values of the universal constants appearing in the  logarithmic estimates. In the favorable case this give us a local uniform bound of solutions to \eqref{NS} in the adequate space. To explain this let us recall the first  logarithmic estimates.
 For \mbox{$\Phi$}  is defined on  \mbox{$]0,+\infty[\times]0,+\infty[$} one denotes
 $$    \|\psi\|_{K_\Phi}:=\sup_{x\neq y}\Phi\big(|\psi(x)-\psi(y)|, |x-y|\big),
 $$
for every $\psi$ an homeomorphism on $\mathbb R^d$.

 A logarithmic estimate in some  functional Banach space $\mathcal X$ is of the form
$$
\|f{\rm o}\psi\|_{\mathcal X}\leq \big[C_1+C_2\ln(\|\psi\|_{C(\psi)})\big]\|f\|_{\mathcal X},
$$
for any  Lebesgue  measure preserving  homeomorphism  \mbox{$\psi$}. The constants $C_1, C_2$ are of course universal and $C(\psi)$ a constant describing the required regularity of $\phi$.

These estimates arise naturally in the study of  transport PDEs, associated to a free-divergence vector field. Indeed, such a vector field gives rise to a bi-Lipschitz measure preserving flow, which plays a crucial role for solving the transport equation. In \cite{Vishik1} Vishik obtained a logarithmic growth  for the Besov space (\mbox{$\mathcal X=B^0_{\infty,1}$} and Lipshitzian flow) with applications to Euler equation. More recently, the authors have proved a similar for $\mathcal X=\BMO$ and Lipschitz flows \cite{BK} and $\mathcal X= L^p\cap\lb $ \cite{BK2}. In the last case, the flow is not Lipschitz and $\Phi$  is defined   by
    \begin{equation*}
 \Phi(r,s)=\left\{ 
 \begin{array}{ll} \max\big(\frac{1+|\ln(s)| }{ 1+|\ln r | },\frac{ 1+|\ln r | }{1+|\ln(s)| }\big),\quad {\rm if}\quad (1-s)(1-r)\geq 0, \\
 {(1+|\ln s|) }{ (1+|\ln r|) },\quad {\rm if}\quad  (1-s)(1-r)\leq 0.
 \end{array} \right.    
      \end{equation*}
In these result the sharp value of $C_1$ and $C_2$ are not important so no attempt to determine theses values were made. Our conjecture about this issue is:

\begin{Conj}
\label{Conj} In both cases considered in \cite{BK,BK2} the constant $C_1$ can be taken equal to $1$.
\end{Conj}
 We are able to confirm this conjecture only in the $\BMO$-case and $\LMO$-case with a bi-Lipschitz flow $\phi$. More precisely we have the following improvement of a result in \cite{BK} for the composition in $\BMO$.

\begin{theorem} 
\label{thm:compo} 
In $\R^d$, there exists a constant $c:=c(d)$ such that for every function $f\in \BMO$ and every measure-preserving bi-Lipschitz homeomorphism $\phi$, we have
$$ \| f\circ \phi\|_{\BMO} \leq \|f\|_{\BMO} \left[1+c \log(K_\phi) \right],$$
where
$$ K(\phi)=K_\phi := \sup_{x\neq y} \max\left( \frac{|\phi(x)-\phi(y)|}{|x-y|},\frac{|x-y|}{|\phi(x)-\phi(y)|}\right).$$ 
\end{theorem} 

\begin{remark} In \cite{BK}, such result was already obtained with a control by  $c_1\left[1+c \log(K_\phi) \right]$ with an implicit constant $c_1$. The aim here is to improve by proving that $c_1$ may be chosen equal to $1$, which brings an important improvement for when the map $\phi$ converges to the identity or any isometry (which is equivalent to $K_\phi$ converges to $1$).
\end{remark}

As an application, our second result is then the following : 
  \begin{theorem} 
\label{uniform} Take $p\in [1,2)$ and $\alpha>1$ and set ${\mathcal B}_{p,\alpha}:=L^p \cap \LMO$. Then for every  $u_0\in L^2(\mathbb R^2)$ a divergence free vector field such that $\omega_0={\rm rot}u_0\in {\mathcal B}_{p,\alpha}$  there exists $T=T(\|\omega_0\|_{{\mathcal B}_{p,\alpha}})$ and $C_0=C_0(\|\omega_0\|_{{\mathcal B}_{p,\alpha}})$ such that the family $(u^\varepsilon)_{\varepsilon>0}$ of solutions to \eqref{NS} satisfies the following bounds uniformly with respect to $\varepsilon>0$:
$$  
\|u^\varepsilon\|_{L^\infty([0,T],L^2)}+ \|{\rm rot}(u^\varepsilon)\|_{L^\infty([0,T], \mathcal B_{p,\alpha})} \leq C_0.
$$
The same holds for the fractional Navier-Stokes equations \eqref{fracNS}.
\end{theorem}


The remainder of this  paper is organized as follows. In the  next section  we describe some preliminaries about functional spaces and how they appear in the study of 2D Euler equation. Then Theorem \ref{il1} is proved in Section \ref{limit1}. Section \ref{limit2} is devoted to the study of Euler equations with an initial vorticity in $\LMOI$, Theorem \ref{il1-bis}. Then in Section \ref{sec:uniform}, we prove Theorem \ref{uniform} by a discretization scheme.

\section{Definitions and preliminaries on functional spaces} \label{sec:def}

This is a preparatory section in which we recall some definitions of useful functional spaces  and we give some results, we need later. 

\subsection{The scale of $\LMO$ spaces}
 
We first define the $\LMO$ spaces:

\begin{definition} Let $\alpha \in[0,\infty)$ and $f:\R^2\to \R$ be a locally integrable function. We say that $f$ belongs to $\LMO$ if 
$$
\|f\|_{\LMO}:=\sup_{0 <r\le \frac{1}{2}} |\ln {r}|^\alpha \left(\av_{B} \left|f-\av_{B}f \right|^2\right)^{\frac{1}{2}}+ \left(\sup_{|B|=1}\int_{B}|f(x)|^2dx \right)^ {\frac{1}{2}}<\infty,
$$
where  the first supremum is taken over all the balls $B$ of radius $r\leq \frac{1}{2}$. 
\noindent For convenience, for $\alpha=1$ then $\LMO$ is denoted $\LMOI$.
\end{definition}

\begin{remark} As dictated by a variant of John-Nirenberg inequalities (see \cite{FPW,BeMa}), if we replace the $L^2$-control of the oscillations by a $L^p$-control for some $p\in(1,2]$ then we obtain an equivalent norm.
\end{remark}

We also recall the functional space $\lb$, introduced in \cite{BK2}. 
\begin{definition} The $\lb$-norm is defined by 
$$
\|f\|_{\lb}:=\|f\|_{{\rm BMO}}+\sup_{B_1,B_2}\frac{|\av_{B_2}(f)-\av_{B_1}(f)|}{1+\ln\big(\frac{ 1-\ln r_{B_2}  }{1-\ln r_{B_1}  }\big)},
$$
where the supremum is taken aver all pairs of balls  \mbox{$B_1$}   and  \mbox{$B_2$}  in  \mbox{$\R^2$}   with  \mbox{$0<r_{B_1} \leq 1$}  and  \mbox{$2B_2\subset B_1$}.
\end{definition}

\begin{remark} \label{remark} We give here some easy remarks on these spaces:
\begin{itemize}
 \item[a)] These spaces $\LMO$ are Banach spaces;
 \item[b)] If $\alpha <\beta$ then ${\it L}^{\beta}{\it mo} \subset {\it L}^{\alpha}{\it mo}$;
 \item[c)] If $\alpha=0$ then $\LMO$ corresponds to the intersection between $bmo$ (the local $\BMO$ space) and $L^1_{uloc}$ (the space of uniformly locally integrable functions).
 \item[d)] For $\alpha\geq 0$ and $p\in(1,\infty)$, the space $\LMO \cap L^p$ is included in $\BMO\cap L^p$.
 \item[e)] The convolution operator, by a $L^1$-normalized function is a contraction on all these spaces.
\end{itemize}
\end{remark}

\begin{lemma} \label{lemma:injection} For $\alpha>1$, we have the continuous embedding $\LMO \hookrightarrow L^\infty$. The condition $\alpha>1$ is optimal, since there exist non-bounded functions belonging to $\LMOI$.
\end{lemma}

\begin{proof} Let $x$ be a fixed point of $\R^2$ and consider $B(r)=B(x,r)$ the balls centered at $x$. Then for a function $f\in \LMO$, it is well-known that we have for $n\gg 1$
\begin{align*}
 \left|  \av_{B(2^{-n})}f - \av_{B(1)} f \right| & \leq \sum_{k=1}^{n} \left|  \av_{B(2^{-k})}f - \av_{B(2^{-k+1})} f \right| \\
 & \leq \|f\|_{\LMO} \sum_{k=1}^{n} (1+ k)^{-\alpha}.
\end{align*}
Since $\alpha>1$ then the sum is convergent and so we deduce that
$$ \limsup_{n\to \infty}  \left|  \av_{B(2^{-n})}f \right| \lesssim \|f\|_{\LMO}.$$
Since $f$ is locally integrable, the differentiation theorem allows us to conclude that $f\in L^\infty$ and
$$ \|f\|_{L^\infty} \lesssim \|f\|_{\LMO}.$$

For the sharpness of the result, we refer to \cite[Proposition 2]{BH} where the function $x\mapsto \log(1-\log(|x|) {\bf 1}_{|x|\leq 1}$ is shown to belong to $\LMOI$, in ${\mathbb R}^2$.
\end{proof}

We also recall a result, proved in \cite[Theorem 1.1]{Pfetre}:
\begin{proposition} \label{prop:riesz}
For every $\alpha\geq 0$ and $p\in(1,\infty)$, the space $\LMO\cap L^p$  is stable by the action of any Riesz transforms.
\end{proposition}

We do not write the proof, it is essentially the same than the one of \cite{Pfetre} excepted that we work here with the local version of $\BMO$-type spaces. The big balls (ball of radius larger than $1$) can be easily studied using the $L^p$ norm.

\medskip

In the sequel we will use  the following  interpolation lemma\footnote{The main point in this lemma is the  linear dependence of the interpolation  constant. Actually, the interpolation itself is well known \cite{GR} but we haven't found  in the literature this type of constants. by sake of completeness we  give the proof.}.

\begin{lemma} 
\label{interpo} 
There exists $C=C(n)>0$ such that the following estimate holds for every \mbox{$r\in [2,+\infty)$} and  every 
smooth function $f$
$$
 \|f\|_{L^r}\leq C r  \|f\|_{L^2\cap BMO}.
 $$
\end{lemma}
\begin{proof} We consider the usual Hardy-Littlewood maximal operator:
$$
M(f)(x)=\sup_{B\ni x}\frac1{|B|}\int_B|f(x)|dx=\sup_{B\ni x} \av_B |f|.
$$
Let $\lambda>0$ and
$$
E_\lambda=\{x: M(f)(x)>\lambda\}.
$$
Let $(Q_i)_i$ be a Whitney covering of $E_\lambda$. We have in particular
$$
Q_i\subset E_\lambda \quad \textrm{and} \quad  4Q_i\cap E_\lambda^c\neq\emptyset,
$$
which implies
$$
m_{4Q_i}(f)\leq\lambda,\qquad \av_{Q_i}(f)\leq4\lambda.
$$
One has, for every $i\in \mathbb N$,
\begin{eqnarray*}
\big(\int_{Q_i}|f(x)|^rdx\big)^{1/r}&\leq&  \big(\int_{Q_i}|f(x)-\av_{Q_i}(f)|^rdx\big)^{1/r}+\lambda|Q_i|^{1/r}
\\
&\leq& |Q_i|^{1/r}\big( \|f\|_{\BMO_r}+\lambda\big)
\\
&\leq& 2|Q_i|^{1/r} \|f\|_{\BMO_r},
\end{eqnarray*}
where $\BMO_r$ is the $\BMO$-norm with oscillations controlled in $L^r$. Summing on $i\in \mathbb N$ one gets
\begin{eqnarray*}
\int_{E_\lambda}|f(x)|^rdx&\leq&  \sum_i\int_{Q_i}|f(x)|^rdx
\\
&\lesssim & C^r|E_\lambda| \|f\|_{\BMO_r}^r,
\end{eqnarray*}
where $C$ is a universal constant.
But, by maximal theorem (see \cite{GR} for instance),
$$
|E_\lambda|\lesssim{\lambda^{-2}}\|f\|_{L^2}^2.
$$
This gives,
$$
\int_{E_\lambda}|f(x)|^rdx\leq {C^r}{\lambda^{-2}}\|f\|_{L^2}\|f\|_{\BMO_r}^r . 
$$
Trivially on has
$$
|f(x)|\leq M(f)(x)\leq \lambda,\qquad \forall\, x\in  E_\lambda^c.
$$
This yields, via H\^older inequality,
\begin{eqnarray*}
\int_{E_\lambda^C}|f(x)|^rdx&\leq& (\int_{E_\lambda^C}|f(x)|^2dx)\|  f\|_{ L^\infty(E_\lambda^c)}^{r-2 }
\\
&\lesssim & \|f\|_{L^2}^2 \lambda^{r-2}.
\end{eqnarray*}
Finally,
\begin{eqnarray*}
\int|f(x)|^rdx\leq \|f\|_{L^2}^2\lambda^{-2}   [C\lambda^r+C^r\|f\|_{\BMO_r}^r],
\end{eqnarray*}
where $C$ is a universal constant.
Taking $\lambda= \|f\|_{\BMO_r}$ we infer
\begin{eqnarray*}
\int|f(x)|^rdx\lesssim  \|f\|_{L^2}^2\|f\|_{\BMO_r}^{r-2}C^r.
\end{eqnarray*}Thus,
$$
\|f\|_{L^r}\lesssim  C\|f\|_{L^2}^{\frac2r} \|f\|_{\BMO_r}^{1-\frac2r}.
$$
One of the direct consequence of   John-Nirenberg inequality is (with $Gamma$-funcion satisfying $\Gamma(r)\lesssim r^r$)
\begin{eqnarray*}
\|f\|_{\BMO_r(\mathbb R^n)}&\lesssim_n&(r\Gamma(r))^{\frac1r} \|f\|_{\BMO(\mathbb R^n)}
\\
&\lesssim_n&r\|f\|_{\BMO(\mathbb R^n)}.
\end{eqnarray*}
Thus , we obtain finally
\begin{eqnarray*}
\|f\|_{L^r}&\lesssim_n&  r^{1-\frac2r}\|f\|_{L^2}^{\frac2r} \|f\|_{\BMO}^{1-\frac
2r}
\\
&\leq &C r\|f\|_{L^2}^{\frac2r} \|f\|_{\BMO}^{1-\frac
2r},
\end{eqnarray*}
 as claimed.
\end{proof}

\subsection{Regularity estimates on the flow for $\LMO$ vorticity}

We first aim to obtain informations on the regularity of the velocity vector-field $u$, associated to a $\LMO$-vorticity $\omega$ via the Biot-Savart law:
\begin{equation}
u=K\ast\omega,\quad \hbox{with}\quad K(x)=\frac{x^\perp}{2\pi|x|^2}. \label{eq:biot-savart}
\end{equation}

We first give a refinement of \cite[Proposition 5]{BH}:

\begin{definition} We say that a function $f:\R^2 \to \R^2$ belongs to the class $L^\beta L$ for $\beta\in[0,1]$ if 
$$ \|f\|_{L^\beta L}:=\sup_{0<|x-y|<\frac12} \ \frac{|f(x)-f(y)|}{|x-y|\big|\ln|x-y|\big|^\beta}+\|f\|_{L^\infty}<\infty.  $$
Note also that the space $L^\beta L$ may also be equipped with the following equivalent norm:
 $$ \|f\|_{L^\beta L}\simeq \sup_{x\neq y} \ \frac{|f(x)-f(y)|}{|x-y|(1+\big|\ln|x-y|\big|^\beta)}+\|f\|_{L^\infty}.  $$
For $\beta=0$, this corresponds to bounded and Lipschitz functions and so $L^0L$ will be denoted by $\Lip$.
\end{definition}

\begin{proposition} \label{prop:velocity} For $p\in(1,2)$, there exists a constant $\rho$ such that for every $\alpha\geq 0$,  and every vorticity $\omega\in \LMO\cap L^p$ the corresponding velocity $u$ given by (\ref{eq:biot-savart}) satisfies the following:
\begin{itemize}
 \item If $\alpha\in[0,1)$ then $u\in L^{1-\alpha}L$ with
 $$ \|u\|_{L^{1-\alpha}L} \leq  \frac{\rho}{1-\alpha} \|\omega\|_{\LMO \cap L^p};$$
\item If $\alpha>1$ then $u\in \Lip$ with
$$ \|u\|_{\Lip} \leq  \frac{\rho \alpha }{\alpha-1} \|\omega\|_{\LMO\cap L^p}.$$
\end{itemize}
\end{proposition}

\begin{proof} The $L^\infty$-norm of $u$ can be more easily bounded.
Actually, a direct consequence of the  Biot-Savart law is
\begin{eqnarray*}
\|u\|_{L^\infty}&\leq& \|K\mathbb 1_{|x|<1}\|_{ L^p}\|\omega(t)\|_{L^{p'}}+ \|K\mathbb 1_{|x|\geq 1}\|_{ L^{p'}} \|\omega(t)\|_{L^p}
\\
&\leq& C_p(\|\omega(t)\|_{\LMO}+\|\omega(t)\|_{L^p}),
\end{eqnarray*}
where $p'$ is the conjugate exponent of $p$ and where we used (since $1<p<2$) that $L^{p'} \subset L^p \cap\BMO \subset L^p \cap \LMO$.

For $\alpha\in[0,1)$, we follow the same proof as in \cite[Proposition 1\&5]{BH} with following the behavior on $\alpha$ (more precisely, we use that $\sum_{n=1}^{N} n^{-\alpha} \lesssim \frac{1}{1-\alpha} N^{1-\alpha}$).

For $\alpha>1$, consider $\omega \in L^\alpha mo$. Then denote by $S_{0}, (\Delta_{n})_{n\geq 0}$ the standard Littlewood-Paley projectors. The following inequality holds (see \cite{BH}):
$$ \|\Delta_{n} \omega\|_{L^\infty} \lesssim (1+n)^{-\alpha}\|\omega\|_{ \LMO}  .$$
Consequently, following \cite[Proposition 5]{BH}
$$ \left|u(x)-u(y)\right| \leq |x-y| \|\nabla S_{0}(u)\|_{L^\infty} + |x-y| \sum_{n\geq 0} (1+n)^{-\alpha}.$$
Invoking Bernstein inequality and the well-known \footnote{We recall that $\nabla u={\mathcal R}(\omega)$ for some Riesz transform ${\mathcal R}$.  }$ \|\nabla \Delta_n u\|_{L^\infty}\simeq \|\Delta_n \omega\|_{L^\infty}$ 
\begin{eqnarray*} \left|u(x)-u(y)\right| &\leq& |x-y| \|\nabla S_{0}(u)\|_{L^\infty} + |x-y| \sum_{n\geq 0} \|\nabla \Delta_n u\|_{L^\infty}
\\
&\leq &  |x-y| \|u\|_{L^\infty}+ |x-y| \sum_{n\geq 0} \|\Delta_n \omega\|_{L^\infty}
\\
&\leq & |x-y| \big(\|u\|_{L^\infty}+ \|\omega\|_{ \LMO }\sum_{n\geq 0} (1+n)^{-\alpha} \big).
\end{eqnarray*}
The first part of the proof and the easy fact
$
\sum_{n\geq 0} (1+n)^{-\alpha} \simeq \frac{ \alpha }{\alpha-1}$ conclude the proof.
\end{proof}

Then, associated to a time-dependent divergence-free vector-field $u:=\R^+ \times \R^2 \to \R^2$, we define the flow $\psi(t,\cdot)$ as the solution of the differential equation,
$$
\partial_t{\psi}(t,x)=u(t,\psi(t,x)),\qquad {\psi}(0,x)=x.
$$
We have the following regularity:

\begin{proposition}{\cite[Proposition 6]{BH}}  \label{proposition:flow} Let $u$ be a smooth divergence-free vector field and $\psi$  be its  flow (and $\psi^{-1}$ its inverse). Then there exists a constant $\eta$ (independent on $u$) such that for every non-increasing function $\alpha:\R^+ \to (0,1]$ and for every $t\geq 0$ we have
$$ |x-y|\neq 0 \Longrightarrow |\psi^{\pm1}(t,x)-\psi^{\pm1}(t,y)|\leq |x-y| e^{\eta V(t)|\ln|x-y||^{1-\alpha(t)}}.$$
where
$$ V(t):=\int_0^t \|u(\tau)\|_{L^{1-\alpha(\tau)} L}d\tau. $$
\end{proposition}

We do not repeat the proof, since it is exactly the same one as detailed for \cite[Proposition 6]{BH}, where the implicit constants are shown to be independent on $\alpha$.

Similarly we have the same for a Lipschitz-velocity, which is more well-known:

\begin{proposition}  \label{proposition:flowlip} Let $u$ be a smooth divergence-free vector field and $\psi$  be its  flow (and $\psi^{-1}$ its inverse). Then there exists a constant $\eta$ (independent on $u$) such that for every $t\geq 0$ we have
$$  \|\psi^{\pm1}(t,\cdot)\|_{Lip} \leq e^{V(t)},$$
where
$$
   V(t):=\int_0^t \|u(\tau)\|_{Lip}d\tau.
$$
\end{proposition}

\section{Inviscid limit for an initial vorticity in $\lb$, Theorem \ref{il1} } \label{limit1}

\begin{proof}[Proof of Theorem \ref{il1}] This proof, which follows a rather classical scheme\footnote{ See  \cite{Ch3} for instance.}, is based on two main ingredients: the control of the BMO-norm of the solution of (E) (proved in \cite{BK2}) and the refined expression of the constant appearing in Lemma \ref{interpo}.

It is well-known  since \cite{Leray} that the bidimensional Navier-Stokes system (\ref{NS}) with initial velocity in $L^2$ has a unique solution $u^\epsilon$ satisfying:
$$
\|u^\varepsilon(t)\|_{L^2}^2+2\varepsilon\int_0^t\|\nabla u^\varepsilon(t')\|_{L^2}^2dt'=\|u_0\|_{L^2}^2,\qquad \forall \, t\geq 0.
$$
The vorticity  $\omega^{\varepsilon}:=\partial_{1}u^{\varepsilon}_2-\partial_{2}u^{\varepsilon}_1$ satisfies the following reaction-diffusion equation
   $$
  \partial_{t}\omega^{\varepsilon}+u^{\varepsilon}\cdot\nabla\omega^{\varepsilon}-\varepsilon\Delta\omega^{\varepsilon}=0,\qquad \omega^\varepsilon_{\mid t=0}= \omega_0.
     $$
The classical $L^p$ estimate the this equation yields 
$$
\|\omega^{\varepsilon}(t)\|_{L^p}\leq \|\omega_0\|_{L^p},\qquad \forall \, t\geq 0.
$$

Let $U^\varepsilon=u^\varepsilon-u$ and $\pi^\varepsilon=P^\varepsilon-P$. One denotes also
$ \Omega^\varepsilon=\omega^\varepsilon-\omega$,
where $\omega^\varepsilon$ is the vorticity of $u^\varepsilon$ and  $\omega$ is the vorticity of $u$.

The vector field $U^\varepsilon$ satisfies
\begin{equation*}
 \left\{ 
\begin{array}{ll} 
\partial_t U^\varepsilon+u^\varepsilon\cdot\nabla U^\varepsilon+\nabla \pi^\varepsilon=U^\varepsilon\cdot\nabla u+\varepsilon\Delta u^\varepsilon,\qquad x\in \mathbb R^2, t>0, \\
\nabla.U^\varepsilon=0,\\
U^\varepsilon_{\mid t=0}= 0.
\end{array} \right.  
 \end{equation*} 
The energy estimate gives
\begin{eqnarray*}
\frac{d}{dt}\| U^\varepsilon\|^2_{L^2}&\leq&
 |\langle U^\varepsilon\cdot\nabla u, U^\varepsilon\rangle|+\varepsilon \|\nabla u^\varepsilon\|_{L^2}\|\nabla U^\varepsilon\|_{L^2}
 \\
 &\leq& I+II.
\end{eqnarray*}
    By $L^2$-continuity of Riesz-operator one has, for every $t\in [0,T]$,
   \begin{eqnarray*}
II&\leq&\varepsilon \|\omega^\varepsilon\|_{L^2}(\|\omega^\varepsilon\|_{L^2}+\|\omega\|_{L^2})
\\
&\leq&\varepsilon C_0.
\end{eqnarray*}
The last estimate follows from the uniform bound of the $L^2$ norm of the vorticities\footnote{By interpolation between $L^p$ and $\BMO$ we know that 
$\omega_0\in L^r$ for every $r\in [p,+\infty[$.}.
 
On the other hand, by H\"older inequality and the continuity of the Riesz-operator one gets, for every $q\geq 2$,
 \begin{eqnarray*}
I&\leq& \int_{\mathbb R^2} |\nabla u(t,x)||U^\varepsilon(t,x)|^2dx
\\
&\leq& \|\nabla u\|_{L^q} \|U^\varepsilon\|_{L^{2q'}}^{2}.
\end{eqnarray*} 
Using Lemma \ref{interpo} we infer
 \begin{eqnarray*}
I\lesssim q \|\nabla u\|_{L^2\cap\BMO} \|U^\varepsilon\|_{L^{2q'}}^{2}.
\end{eqnarray*} 
The continuity of the Riez operator on $L^2\cap\BMO$ yields\footnote{The continuity of a Riesz operator on $\BMO$ was proved in \cite{Pfetre}, see also Proposition \ref{prop:riesz}.}
\begin{eqnarray*}
I&\lesssim&  q \|\omega(t)\|_{L^2\cap\BMO} \|U^\varepsilon(t)\|_{L^{2q'}}^{2}
\\
&\leq &qC_0e^{C_0t}\|U^\varepsilon(t)\|_{L^{2q'}}^{2},
\end{eqnarray*} 
where we have used Theorem 1.1 in \cite{BK2} ($C_0= C_0(\|\omega_0\|_{L^p\cap\lb}$).
Using H\"older inequality and Biot-Savart law one obtains 
\begin{eqnarray*}
\|U^\varepsilon\|_{L^{2q'}}^{2}&\lesssim& \|U^\varepsilon\|_{L^{\infty}}^{\frac2q} \|U^\varepsilon\|_{L^{2}}^{2-\frac2q }
\\
&\lesssim &   \|\Omega^\varepsilon\|_{L^p\cap L^{3}}^{\frac2q} \|U^\varepsilon\|_{L^{2}}^{2-\frac2q }
\\
&\lesssim &  ( \|\omega^\varepsilon\|_{L^p\cap L^{3}}+ \|\omega\|_{L^p\cap L^{3}})^{\frac2q} \|U^\varepsilon\|_{L^{2}}^{2-\frac2q }.
\end{eqnarray*} 
Since $\|\omega^\varepsilon(t)\|_{L^p\cap L^{3}}^{\frac2q}$ is uniformly bounded then the outcome is
\begin{eqnarray}
\frac{d}{dt}\| U^\varepsilon(t)\|^2_{L^2} &\leq& C_0\big(\varepsilon+qe^{C_0t} \|U^\varepsilon\|_{L^{2}}^{2-\frac2q }\big), 
\label{eq:dif}
\\
&\leq& C_0\big(\varepsilon+q e^{C_0t}\|U^\varepsilon\|_{L^{2}}^{2-\frac2q }\big), \nonumber
\end{eqnarray} 
for all $q\geq 2$ and some constant $C_0=C_0(\|\omega_0\|_{L^p\cap\lb})$.  

Take $g^\varepsilon(t):=\| U^\varepsilon(t)\|^2_{L^2}$ and define $T^\varepsilon<T$ the maximal time:
$$
T^\varepsilon:=\max\{ t\leq T:   \sup_{\tau\in [0,t]}g^\varepsilon(\tau)\leq \frac{1}{e^2}\}.
$$
For every $t\in (0,T^\varepsilon)$ one chooses $ q=-\ln( g^\varepsilon(t))$, in (\ref{eq:dif}) to get
\begin{eqnarray*}
\dot{g}^\varepsilon(t)&\leq& C_0\big(\varepsilon-e^{C_0t}\ln(g^\varepsilon(t))g^\varepsilon(t)\big).
\end{eqnarray*}
Integrating this differential inequality 
\begin{eqnarray*}
{g}^\varepsilon(t)&\leq& C_0\varepsilon t+\int_0^t-C_0e^{C_0t"}\ln(g^\varepsilon(t"))g^\varepsilon(t")dt"
\\
&\leq& C_0\varepsilon t'+\int_0^t-C_0e^{C_0t"}\ln(g^\varepsilon(t"))g^\varepsilon(t")dt".
\end{eqnarray*}
for every $t\leq t'<T^\varepsilon$.

Assuming $C_0T\varepsilon_0<1$ and applying  Lemma \ref{lemma:Osgood} below
\begin{eqnarray*}
-\ln(-\ln({g}^\varepsilon))  +\ln(-\ln(C_0 t' \varepsilon))  \leq (e^{C_0t}-1),\qquad \forall t\leq t'<T^\varepsilon.
\end{eqnarray*}
This yields, for all $t\leq t'<T^\varepsilon$
\begin{eqnarray*}
{g}^\varepsilon(t)\leq (C_0 t'\varepsilon)^{\beta(t)},
\end{eqnarray*}
with $\beta(t)= \exp(1-e^{C_0t}).$
In particular,
\begin{eqnarray*}
{g}^\varepsilon(t)\leq  (C_0 t\varepsilon)^{\beta(t)},\qquad \forall t\in [0,T^\varepsilon[.
\end{eqnarray*}
If we assume that $\varepsilon_0$ satisfies also 
$$
(C_0 T \varepsilon_0)^{\beta(T)}\leq \frac1{e^2}.
 $$
 we get 
 $
 T^\varepsilon=T,$  for all $\varepsilon\leq \varepsilon_0$.
 
 This gives finally   and so
  \begin{eqnarray*}
{g}^\varepsilon(t)\leq  (C_0 t\varepsilon)^{\beta(t)},\qquad \forall t\in [0,T], \varepsilon\leq \varepsilon_0,
\end{eqnarray*}
for some constant $C_0=C_0(\|\omega_0\|_{L^p\cap\lb})$  and $\varepsilon_0=\varepsilon_0(\|\omega_0\|_{L^p\cap\lb},T)$ as claimed.
\end{proof}

The following Osgood Lemma  is a slight generalization of \cite[Lemma 3.4]{bcd} for which the function $c$ is constant and its  proof is an easy application of  it.

 \begin{lemma}[Osgood lemma] \label{lemma:Osgood} Let $\rho$ be a measurable function from $[t_0, T ]$ to $[0, a]$, $\gamma$ a locally
integrable function from $[t_0, T ]$ to $\mathbb R^+$, and $\mu$ a continuous and nondecreasing function from $[0,a]$ to $\mathbb R^+$. Assume that, for some nonnegative nondecreasing continuous $c$, the function $\rho$ satisfies
$$
 \rho(t)\leq c(t)+\int^t_{t_0}\gamma (t')\mu(\rho(t'))dt'.
 $$
 Then 
 $$
 -\mathcal M(\rho(t))+\mathcal M(c(t))\leq \int^t_{t_0}\gamma (t')dt'.
 $$
 with
 $$
 \mathcal M(x)=\int^a_x\frac1{\mu(r)}dr.
 $$
\end{lemma}

\section{Inviscid limit for an initial vorticity in $\LMOI$} \label{limit2}

\subsection{Existence and uniqueness of global solution for Euler equation with an initial vorticity in $\LMOI$}

In this section we will use Propositions \ref{prop:velocity} and \ref{proposition:flow} to prove that if we solve the 2D Euler equations with initial vorticity $\omega_{0}\in \LMOI=L^{1} mo$ then given $\delta>0,$ $\omega(t) \in L^{\alpha(t)}$ where $\alpha(t)$ is a continuous function with $\alpha(0)=1$ and $\alpha(t) \geq 1-\delta$ for all $0<t<\infty.$ Prior to stating the precise theorem, we will make a few comments on the previous results in this direction. 
It was proven in Vishik \cite{Vishik1} that if $\omega_{0}$ satisfies 
\begin{equation} \sum_{-1}^{n} \|\Delta_{j} \omega_{0} \|_{L^\infty} \lesssim \Pi(n), \label{3.1} \end{equation} 
where $\Pi$ is an increasing function with $\int_{2}^{\infty} \frac{dn}{n\Pi(n)}=\infty$ then we can solve the Euler equations with for every $t>0$ 
\begin{equation} \sum_{-1}^{n} \| \Delta_{j}\omega(t) \|_{L^\infty} \lesssim  n\Pi(n), \label{3.2} \end{equation} 
where the constant gets worse in time depending upon $\omega_{0}.$ 

In particular, this result in Besov spaces flavor proves some propagation of the initial regularity but with a loss.

Let us then consider the space $\LMOI$. It is easy to see that for $\omega_0\in \LMOI$ then (\ref{3.1}) is satisfied with $\Pi(n)=\log(n)$. Applying Vishik's result gives us a solution of Euler equations satisfying (\ref{3.2}). We claim that indeed the solution is better and satisfies for $t>0$ and any $\delta>0$
\begin{equation} \sum_{-1}^{n} \| \Delta_{j}\omega(t) \|_{L^\infty} \lesssim  n^{\delta} \label{3.2-bis} \end{equation} 
(still with implicit constants depending on time and on $\omega_0$). This will be a consequence of the following Theorem (since $\omega(t)\in L^{1-\delta} mo$ implies (\ref{3.2-bis})).
So by this way, $\LMOI$ appears as a subspace of vorticities satisfying (\ref{3.1}) with $\Pi=\log$ where we improve Vishik's result and the loss of regularity is as small as we want (in terms of exponent of $n$ in (\ref{3.2-bis}), improving (\ref{3.2})).

We now state the main theorem of this section. 

\begin{theorem} \label{thm:3.1}
Let $0<\delta<1$ and $p\in[1,2)$ be given. Suppose that $\omega_{0} \in \LMOI \cap L^p$. Then there exists a unique solution of the 2D incompressible Euler equations such that $\omega \in L^\infty _{\text{loc}}( [0,\infty);L^{\alpha}mo \cap L^p),$ where $$\alpha(t)=1-\sqrt{t} , \, \, 0\leq t \leq \delta^2,$$
$$=1-\delta, \, \, t>\delta^2.$$ 
Moreover, for some constant$C=C(\omega_0)$, we have
$$ \|\omega(t)\|_{L^{\alpha(t)}mo \cap L^p} \leq C_0e^{C_0 t}.$$
\end{theorem}     

\begin{remark} 
\begin{itemize}
\item First, $\omega \in L^\infty _{\text{loc}}( [0,\infty);L^{\alpha}mo \cap L^p)$ means that for every $T>0$
$$ \sup_{t\in[0,T]} \|\omega(t)\|_{L^{\alpha(t)}mo \cap L^p} <\infty.$$
\item We can take any function $\alpha$ of the form
$$\alpha(t)=1-t^{\rho} , \, \, 0\leq t \leq \delta^{1/\rho},$$
$$=1-\delta, \, \, t>\delta^{1/\rho},$$ 
for some $\rho\in(0,1)$ and $\delta>0$.
 \end{itemize}
The reason that $\rho=1$ is not admissible is that the regularity loss must be enough so that $\frac{1}{1-\alpha(t)}$ is integrable near $t=0,$ as will be clear from the proof. We do not believe that this is an artifact of our proof.  
\end{remark}

To prove Theorem \ref{thm:3.1}, we will rely upon Propositions \ref{prop:velocity} and \ref{proposition:flow} as well as the following important proposition:  

\begin{proposition} \label{prop:3.2-bis}
Let fix $p\in[1,2)$, $\alpha\in(0,1)$ and $\psi$ an homeomorphism (preserving the measure) such that for every $x\neq y$ 
$$ |\psi^{\pm1}(x)-\psi^{\pm1}(y)|\leq |x-y| e^{V|\ln|x-y||^{1-\alpha}},$$
for some constant $V$.
Then for every $\omega_0 \in \LMO \cap L^p$
\begin{equation} \|\omega_0(\psi^{-1}) \|_{\LMO\cap L^p} \lesssim (1+V) \|\omega_0\|_{\LMO\cap L^p}. \end{equation}
\end{proposition}

We then easily deduce the following corollary:

\begin{corollary} \label{coro:3.2}
Let fix $p\in[1,2)$ and $\psi$ an homeomorphism (preserving the measure) such that for every $x\neq y$ 
$$ |\psi^{\pm1}(x)-\psi^{\pm1}(y)|\leq |x-y| e^{V|\ln|x-y||^{1-\alpha}},$$
for some constants $V,\alpha$.
Then for every $\omega_0 \in \LMOI \cap L^p$,
\begin{equation} \|\omega_0(\psi^{-1}) \|_{\LMO\cap L^p} \lesssim (1+V) \|\omega_0\|_{\LMOI\cap L^p}. \label{3.3}\end{equation}
\end{corollary}

We only prove Proposition \ref{prop:3.2-bis}. 

\begin{proof} Since $\psi$ preserves the measure, the $L^p$ norm is conserved. Hence, we only have to deal with the homogeneous part of the $\LMO$-norm.
 Let $B$ be a ball of radius $r\leq \frac{1}{2}$ then
\begin{align*}
 \av_{B} \left|f-\av_{B}f \right| & \leq  \left(\av_{B} \left|f-\av_{B}f \right|^2\right)^{\frac{1}{2}} \\
 & \leq \inf_{C} \left(\av_{B} \left|f-C \right|^2 \right)^{\frac{1}{2}}, 
\end{align*}
where the infimum is taken over all the constants $C>0$.

Applying this inequality for $f=\omega=\omega_0(\psi^{-1})$, it comes for every $q\geq 2$ and every $C>0$,
\begin{align*}
 \av_{B} \left|\omega-\av_{B}\omega \right| & \leq \left(\av_{B} \left| \omega_0(\psi) -C \right|^q \right)^{\frac{1}{q}} \\
 &  \leq  \left(\av_{\psi(B)} \left| \omega_0 -C \right|^q \right)^{\frac{1}{q}}.
\end{align*}
Due to the modulus regularity of $\psi$, if $B$ is a ball of radius $r$ then $\psi(B)$ is included in $\tilde{B}$ a ball of radius 
$$ \tilde{r} := r e^{V|\ln r|^{1-\alpha}}.$$
So, for every $C>0$
\begin{align*}
 \av_{B} \left|\omega-\av_{B}\omega \right| &  \leq  \left(\frac{|\tilde{B}|}{|B|}\right)^{\frac{1}{q}} \left(\av_{\tilde{B}} \left| \omega_0 -C \right|^p \right)^{\frac{1}{q}} \\
 & \leq  \left(\frac{\tilde{r}}{r}\right)^{\frac{2}{q}} \left(\av_{\tilde{B}} \left| \omega_0 -C \right|^q \right)^{\frac{1}{q}} \\
 & \leq e^{\frac{2}{q} V|\ln r|^{1-\alpha}} \left(\av_{\tilde{B}} \left| \omega_0 -C \right|^q \right)^{\frac{1}{q}}.
\end{align*}
Then we may chose $C=\av_{\tilde{B}} \omega_0$ and using the $\LMO_q$ regularity of $\omega_0$, we obtain
\begin{align*}
 \av_{B} \left|\omega-\av_{B}\omega \right| \leq e^{\frac{2}{q}V|\ln r|^{1-\alpha}} |\ln {r}|^{-\alpha} \|\omega_0\|_{\LMO_q},
\end{align*}
where $\LMO_q$ is the $\LMO$-space equipped with the equivalent norm involving oscillations in $L^q$.
Using the John-Nirenberg inequality, we know that\footnote{We note that using Proposition \ref{prrp}, this inequality may be weakened with a growth of order $q^\delta$ for $\delta>0$ and maybe just some logarithmique growth on $q$. Unfortunately, this improvment does not really help to get around the (as small as we want) loss of regularity from the initial condition $\omega \in \LMOI$ and the solution. We just point out that taking into account this improvment, the solution can be shown to live into a Morrey-Campanato space smaller than $\LMO$ with only a $\log-\log$ loss of regularity. Without details, we could bound the oscillation on a ball of radius $r$ by $\frac{\log(|\log(r)|)}{|\log(r)|}$ instead of $|\log(r)|^{\delta-1}$ as we are doing here.}
$$ \|\omega_0\|_{\LMO_q} \lesssim q \|\omega_0\|_{\LMOI}$$
which yields
\begin{align*}
 \av_{B} \left|\omega-\av_{B}\omega \right| \lesssim q e^{\frac{2}{q}V|\ln r|^{1-\alpha}} |\ln {r}|^{-\alpha} \|\omega_0\|_{\LMO}.
\end{align*}
Optimizing in $q\geq 2$ (which means to chose $q= 2(V+1)|\ln r|^{1-\alpha}$) gives
\begin{align*}
 \av_{B} \left|\omega-\av_{B}\omega \right| \lesssim (1+V)|\ln r|^{1-\alpha} |\ln {r}|^{-1} \|\omega_0\|_{\LMOI}.
\end{align*}
Hence
\begin{align*}
 |\ln(r)|^{\alpha} \av_{B} \left|\omega-\av_{B}\omega \right| \lesssim c(1+V) \|\omega_0\|_{\LMOI}.
\end{align*}
\end{proof}

Before proving Theorem \ref{thm:3.1}, we just state the following lemma (which is a variant of standard Gronwall Lemma):

\begin{lemma}
Let $f$ be a smooth function defined on $[0,1]$ satisfying the following inequality:
\begin{equation} f(t) \leq A+B\int_{0}^{t}\frac{1}{\sqrt{s}} f(s)ds. \end{equation}

Then, $f$ satisfies the following (sharp) a-priori growth estimate:

\begin{equation} f(t) \leq A \exp(2B\sqrt{t}). \end{equation}
\end{lemma} 

\begin{proof}[Proof of Theorem \ref{thm:3.1}]
Now, use Propositions \ref{prop:velocity}, \ref{proposition:flow} and Corollary \ref{coro:3.2} and we get the following a-priori estimate: 

\begin{equation} \|\omega(t)\|_{L^{\alpha(t)}mo \, \cap L^2} \lesssim \|\omega_{0}\|_{\LMOI\cap \, L^2} \left(1+C\int_{0}^{t} \frac{1}{1-\alpha(s)} \|\omega(s)\|_{L^{\alpha(s)}mo\, \cap L^2} \, ds \right). \label{3.4} \end{equation}
We are free to choose $\alpha(t)$ as we wish in order to get something useful out of the previous inequality. We wish to choose $\alpha$ so that $\alpha(0)=1.$ However, in order that inequality (\ref{3.4}) not be an empty inequality, we will need $\alpha(t)$ to decrease very sharply near $t=0$ in such a way that $\frac{1}{1-\alpha(t)}$ is integrable near $t=0.$ To simplify things, we will define $\alpha(t)$ in the following way:
$$\alpha(t)=1-\sqrt{t}, \quad 0\leq t\leq \delta^2,$$
$$\alpha(t)=1-\delta,  \quad t>\delta^2.$$
Note that $\alpha$ is continuous on $[0,\infty).$

Using the previous Lemma in conjunction with estimate (\ref{3.4}), we see that $\omega$ satisfies the following a-priori estimate on $[0,\delta^2]:$
\begin{equation} \|\omega(t)\|_{L^{\alpha(t)} mo \cap L^2} \leq \|\omega_{0}\|_{\LMOI} \exp(C\delta \|\omega_0\|_{\LMOI}), \, \, t\in [0,\delta^2]. \label{4} \end{equation}
To control $\omega(t)$ on $(\delta^2,\infty)$ we are going to use the standard Gronwall lemma as follows. For $t\geq \delta^2,$ estimate (\ref{3.4}) tells us that 
 \begin{align} 
 \lefteqn{\|\omega(t)\|_{L^{1-\delta}mo \, \cap L^2}} & & \nonumber \\
 & &  \lesssim \|\omega_{0}\|_{\LMOI\cap \, L^2} \left(1+C\int_{0}^{\delta^2} \frac{1}{\sqrt{s}} \|\omega(s)\|_{L^{\alpha(s)}mo\, \cap L^2} \, ds +C\int_{\delta^2}^{t} \frac{1}{\delta^2} \|\omega(s)\|_{L^{1-\delta}mo\, \cap L^2} \, ds \right). \label{5} \end{align}
Now, the first integral in (\ref{5}) is controlled by (\ref{4}). Therefore, we can apply the standard Gronwall lemma to control $\omega$ in $L^{\infty}_{\text{loc}} ( [0,\infty); L^{\alpha}mo\cap L^2),$ with $\alpha(t)$ chosen as above. In particular, we lose only an arbitrarily small amount of regularity when beginning with data in $\LMOI \cap L^2 .$ 
\end{proof}

\subsection{The inviscid limit when $\omega_{0} \in \LMOI$, Theorem \ref{il1-bis}}

In this section we will prove a sharper result on the rate of convergence in the inviscid limit of the Navier-Stokes equations in Theorem \ref{il1} when the initial data is taken in Lmo. Indeed, in the proof of Theorem \ref{il1}, all we used is an a-priori estimate on $u$ in $\BMO$. However, when we take initial data in $\LMOI$, we will be able to use a-priori estimates on $L^{1-\delta} mo$ for all $\delta>0.$ This fact, coupled with a sharper version of Lemma \ref{interpo} in the $\LMO$ case will allow us to give a better rate than the $\epsilon^{e^{-t}}$ from Theorem \ref{il1}.

In particular, we will be able to prove the following theorem.

\begin{theorem} \label{il2} 
Assume $p\in [1,2)$. Let  $u_0\in L^2(\mathbb R^2)$ a divergence free vector fields such that $\omega_0\in \LMOI \cap L^p$. Then, for every $T>0$ and for every $\delta\in(0,1)$ there exist $C=C(u_0,\delta)$ and $\varepsilon_0=\varepsilon_0(u_0,T,\delta)$ such that
$$
\|u^\varepsilon(t)-u(t)\|_{L^2(\mathbb R^2)}\leq (CT\varepsilon)^{\frac{1}{2}e^{\beta(t)}},\qquad \forall t\in [0,T] ,\quad\forall \varepsilon\leq \varepsilon_0,
$$
with
$$ \beta(t) = \max(1-\delta,(1-\frac{e^{C_0t}-1}{2})^{1\over \delta}).$$
\end{theorem}

To prove this theorem we will rely upon a generalized version of the John-Nirenberg lemma in $\LMO$.

\begin{proposition} \label{prrp}
Let $\alpha\in[0,1)$, then there exists $C_{1}$ and $C_{2}$  depending only upon the dimension and $0\leq \alpha<1$ such that given any cube $Q$ in $\mathbb{R}^{n}$, any function $f\in\LMO$ and any $\lambda>0,$

$$ | \{ x\in Q : |f(x)-\av_{Q} f| >\lambda  \} | \leq C_{1}\exp (-\frac{C_{2}}{\|f\|_{\LMO}} \lambda^{\frac{1}{1-\alpha}}) |Q|.$$
\end{proposition}

The proof of this proposition can be found, for example, in the paper of Caffarelli and Huang \cite[Remark 2.4]{CaH} and in the work of Spanne \cite{Spanne}. The proof is a simple adaptation of the original proof of the John-Nirenberg inequality using the Calder\'{o}n-Zygmund decomposition.  

Based upon this Proposition, we have the following John-Nirenberg inequality (uniformly in $r\gg 1$)
$$
\|f\|_{\LMO_r}\lesssim r^{1-\alpha} \|f\|_{\LMO},
$$
where $\LMO_r$ stands for the $\LMO$-norm with oscillations controlled in $L^r$.
Then, one can use the proof of Lemma \ref{interpo} to prove:

\begin{lemma} There exists $C>0$ depending upon $\alpha\in[0,1)$ such that the following estimate holds for every $p\in [2,+\infty[$ and  every 
smooth function $f$
$$ \|f\|_{L^p}\leq C p^{1-\alpha}  \|f\|_{L^2\cap \LMO}. $$
\end{lemma}

\begin{proof}[Proof of Theorem \ref{il2}] We recall that $\alpha$ can be chosen in $(0,1)$, so we will fix it later according to $\delta \in(0,1)$.
Theorem \ref{il2} then follows from the proof of Theorem \ref{il1}, replacing Lemma \ref{interpo} by the last one. With the same notations, $g^\varepsilon(t):=\| U^\varepsilon(t)\|^2_{L^2}$ satisfies the following differential inequality:
$$ \dot{g}^\varepsilon(t) \leq  C_0\big(\varepsilon+e^{C_0t}|\ln(g^\varepsilon(t))|^{1-\alpha} g^\varepsilon(t)\big). $$
Hence,
$$ {g}^\varepsilon(t) \leq C_0\varepsilon t+\int_0^t C_0e^{C_0t"}|\ln(g^\varepsilon(t"))|^{1-\alpha} g^\varepsilon(t")dt". $$
Then using Osgood Lemma (Lemma \ref{lemma:Osgood}), it comes
$$ -|\ln({g}^\varepsilon)|^\alpha  +|\ln(C_0 t' \varepsilon))|^\alpha  \leq e^{C_0t}-1,\qquad \forall t\leq t'<T^\varepsilon. $$
For every $\delta<1$, it comes for $\epsilon\leq \epsilon_0(u_0,T,\delta)$ small enough
\begin{align*}
|\ln({g}^\varepsilon)|^\alpha & \geq   |\ln(C_0 t' \varepsilon))|^\alpha -  (e^{C_0t}-1)\\
 & \geq  \beta(t)^\alpha|\ln(C_0 t' \varepsilon))|^\alpha,
\end{align*}
with
$$ \beta(t):=  \max(1-\delta,(1-\frac{e^{C_0t}-1}{2})^{1\over \delta}) \leq \left(1 - \frac{e^{C_0t}-1}{|\ln(C_0 t' \varepsilon))|^\alpha}\right)^{1\over \alpha}.$$
We fix $\alpha\in(0,1)$ and $\epsilon_0(u_0,T,\delta)$ according to $\delta$ such that $\beta$ satisfies this previous inequality. Then we conclude by reproducing the same reasoning as for Theorem \ref{il1}, with these slight modifications.
\end{proof}

\section{Uniform estimates for solutions of Navier-Stokes equation with a vorticity in $\LMO$ for $\alpha>1$} \label{sec:uniform}

In this section, we aim to describe more results when we assume that the vorticity is more regular, and more precisely when $\omega_0\in \LMO$ for some $\alpha>1$.

\begin{remark} \label{remark:ulip} First when the velocity $u$ is associated to a $\LMO$ vorticity by the Biot-Savart law (\ref{eq:biot-savart}) then if $\alpha>1$ we deduce by combining Lemma \ref{lemma:injection} and Proposition \ref{prop:riesz} that $u$ is Lipschitz.
\end{remark}

We first aim to prove a slight improvement of results in \cite{BK}, about composition in $\LMO$-spaces by a bi-Lipschitz measure-preserving map.

\subsection{Composition in $\LMO$ by a bi-Lipschitz map} 

\begin{theorem} 
\label{thm:compo-bis} 
In $\R^d$, there exists a constant $c:=c(d)$ such that for every function $f\in \BMO$ and every measure-preserving bi-Lipschitz homeomorphism $\phi$, we have
$$ \| f\circ \phi\|_{\BMO} \leq \|f\|_{\BMO} \left[1+c \log(K_\phi) \right],$$
where
$$ K(\phi)=K_\phi := \sup_{x\neq y} \, \max\left( \frac{|\phi(x)-\phi(y)|}{|x-y|},\frac{|x-y|}{|\phi(x)-\phi(y)|} \right) \geq 1.$$ 

\end{theorem} 

\begin{remark} Let us first point out that this property of $\BMO$ space, is not invariant by changing with an equivalent norm. So the precise statement should be : there exists a norm such that Theorem \ref{thm:compo-bis} holds for $\BMO$ equipped with it.
\end{remark}

\begin{remark} In \cite{BK}, such result was already obtained with a control by  $c_1\left[1+c \log(K_\phi) \right]$ with an implicit constant $c_1>1$. The aim here is to improve by proving that $c_1$ may be chosen equal to $1$, which brings an important improvement for when the map $\phi$ converges to the identity or any isometry (which is equivalent to $K_\phi$ converges to $1$). This improvement will be very important for our purpose in the next subsections, as we will see.
\end{remark}

\begin{proof}
For more convenient, we will consider the norm of $\BMO_2$ based on $L^2$-oscillation.
If $K_\phi \geq 2$ then the desired result was already obtained in \cite{BK} since then
$$ 1+c \log(K_\phi) \simeq \log(K_\phi).$$
So let us focus on the more interesting case, when $K_\phi\in[1,2]$.
Consider such a function $f\in \BMO$ and map $\phi$. Fix a ball $B=B(x_0,r)$ and look for an estimate of the oscillation
$$ \osc(f\circ \phi ,B):= \left(\av_{B}\left| f\circ \phi(x) - \av_B f \circ \phi \right|^2 dx\right)^{\frac{1}{2}}.$$
Then, it is well-known that
$$ \osc(f\circ \phi ,B)= \inf_{C\in \R} \left(\av_{B}\left| f\circ \phi(x) - C \right|^2 dx\right)^{\frac{1}{2}}$$
and so in particular
$$ \osc(f\circ \phi ,B)\leq \left(\av_{B}\left| f\circ \phi(x) - \av_{K_\phi \tilde{B}} f \right|^2 dx\right)^{\frac{1}{2}},$$
where $\tilde{B}:=B(\phi(x_0),r)$ and $K_\phi \tilde{B}$ the dilated ball.
Using the measure preserving property and the fact that $\phi(B) \subset K_\phi \tilde{B}$, it comes
\begin{align*}
\osc(f\circ \phi ,B) & \leq \left(\av_{\phi(B)}\left| f - \av_{K_\phi \tilde{B}} f \right|^2 dx\right)^{\frac{1}{2}} \\
 & \leq (K_\phi)^{d/2} \left(\av_{K_\phi\tilde{B}}\left| f - \av_{K_\phi \tilde{B}} f \right|^2 dx\right)^{\frac{1}{2}} \\
 & \leq K_\phi^{d/2} \|f\|_{\BMO}.
\end{align*}
Since $K_\phi\in[1,2]$, we have
$$ K_\phi^{d/2} = (1+K_\phi-1)^{d/2} \leq 1+c_1(K_\phi-1) \leq 1+c_2\log(K_\phi),$$
for some numerical constants $c_1,c_2$ only depending on the dimension $d$. We conclude to the desired estimate: uniformly with respect to the ball $B$
$$ \osc(f\circ \phi ,B) \leq \left[1+c\log(K_\phi)\right] \|f\|_{\BMO}.$$
\end{proof}

We can also produce a similar reasoning for the $\LMO$ spaces:

\begin{theorem} 
\label{thm:compo2} 
In $\R^d$ with $\alpha>1$, $p\in(1,2]$, there exists a constant $c:=c(d,\alpha,p)$ such that for every function $f\in \LMO \cap L^p$ and every measure-preserving bi-Lipschitz homeomorphism $\phi$, we have
$$ \| f\circ \phi\|_{\LMO \cap L^p} \leq \|f\|_{\LMO \cap L^p} \left[1+c \log(K_\phi)\right],$$
where
$$ K(\phi)=K_\phi := \sup_{x\neq y} \, \max\left( \frac{|\phi(x)-\phi(y)|}{|x-y|},\frac{|x-y|}{|\phi(x)-\phi(y)|} \right) \geq 1.$$ 
Moreover, $\alpha \to c(d,\alpha,p)$ can be chosen increasing on $\R^{+}$.
\end{theorem} 

The importance of the result is the behavior for $\phi$ almost an isometry, which means $K_\phi$ almost equal to $1$.


\begin{proof} Since the case of the logarithmic growth for $K_\phi\geq 2$ was already studied in \cite{BK}, we only focus on the case $K_\phi\in[1,2]$.
We first describe the norm we will consider on $\LMO\cap L^p$ : 
$$ \|f\|_{\LMO \cap L^p} = \|f\|_{\overline\LMO} + \|f\|_{L^p},$$
where the $\overline\LMO$-part is the homogeneous part, obtained by considering $L^1$-oscillations and more precisely:
$$
\|f\|_{\overline\LMO}:=\sup_{0 <r\le \frac{1}{2}} |\ln {r}|^\alpha \inf_{c} \left(\av_{B} \left|f-c \right|\right).
$$
We know that this norm is equivalent to the above defined norm for $\LMO\cap L^p$. So let us work with this norm and write
$$ \osc(f,B):=\inf_{c} \left(\av_{B} \left|f-c \right|\right).$$

First $\Phi$ preserves the measure so $\|f \circ \phi\|_{L^p} = \|f\|_{L^p}$.
Let us consider the same notations as in the previous proof.
So we fix a ball $B=B(x_0,r)$ of radius $r\leq\frac{1}{2}$ and a constant $c$.
If $K_\phi r\leq \frac{1}{2}$ then we just repeat the previous reasoning and we get
\begin{align*}
\osc(f\circ \phi ,B) & \leq (K_\phi)^{d} \osc(f, K_\phi \tilde{B}) \\
& \leq (K_\phi)^{d} |\log(K_\phi r)|^{-\alpha} \|f\|_{\overline\LMO} \leq (K_\phi)^{d} (1+ \log(K_\phi))^\alpha |\log(r)|^{-\alpha} \|f\|_{\overline\LMO},
\end{align*}
 where we used that 
$$ \frac{|\log(r)|}{|\log(r)|-\log(K_\phi)} = 1+ \frac{\log(K_\phi)}{|\log(r)|-\log(K_\phi)} \leq 1+ \frac{\log(K_\phi)}{\log(2)}.$$
Since $K_\phi\in[1,2]$, we have (since $\alpha \geq 1$)
$$ (K_\phi)^{d} (1+ \log(K_\phi))^\alpha \leq (1+ c\log(K_\phi))$$
for some (large enough) numerical constant $c\gg d+\alpha$. We then conclude to 
\begin{align*}
\osc(f\circ \phi ,B) \leq (1+ c\log(K_\phi)) |\log(r)|^{-\alpha} \|f\|_{\overline\LMO}.
\end{align*}

If $K_\phi r\geq \frac{1}{2}$ (which means that $r\geq \frac{1}{4}$) then we know that 
$$ \left| \av_{\phi(B)} |f| - \av_{\tilde{B}} |f| \right|\lesssim (K_\phi-1) \|f\|_{L^\infty}$$
since $\phi(B) \subset K_\phi \tilde{B}$ and $K_\phi^{-1} \tilde{B} \subset \phi(B)$ so that 
$$|\phi(B) \setminus \tilde{B}|+|\tilde{B}\setminus \phi(B)| \lesssim (K_\phi-1) r^d.$$
Due to Lemma \ref{lemma:injection}, we deduce that 
$$ \left| \av_{\phi(B)} |f| - \av_{\tilde{B}} |f| \right|\lesssim (K_\phi-1) \|f\|_{\LMO}.$$
Consequently, it comes
\begin{align*}
\osc(f\circ \phi ,B) & = \inf_{c} \left(\av_{\phi(B)}\left| f - c \right| dx\right) \\
 & \leq \osc(f,\tilde{B}) + C (K_\phi-1) \|f\|_{\LMO} \\
& \leq |\log(r)|^{-\alpha} \|f\|_{\overline\LMO} + C (K_\phi-1) \|f\|_{\LMO},
\end{align*}
where $C$ denotes here a universal constant and may vary from line to another line. Since $r\geq \frac{1}{4}$
we get
\begin{align*}
\osc(f\circ \phi ,B) \leq |\log(r)|^{-\alpha} \left(\|f\|_{\overline\LMO} + C (K_\phi-1) \|f\|_{\LMO}\right).
\end{align*}

Finally, we also obtain that
$$ \|f \circ \phi\|_{L^p \cap \LMO} \leq (1+ C (K_\phi-1)+ c\log(K_\phi)) \|f\|_{L^p \cap \LMO},$$
and we conclude since for $K_\phi\in[1,2]$, $(K_\phi-1) \lesssim \log(K_\phi)$.
\end{proof}

\subsection{Uniform estimates for discretized solution of 2D Navier-Stokes equation} \label{subsec:uniform}

In this paragraph the small parameter $\varepsilon$ in Navier-Stokes equation (\ref{NS}) is fixed. For simplicity we drop the index $\varepsilon$.
We aim to discretize this equation, using the so-called Trotter's formula to combine the two phenomenons : the transport part and the diffusion part.

Let $T>0$  to be chosen later. For every $n\in  \mathbb N^*$ one denote
$$T_{i}^n=i\frac{T}{n}, \qquad i=0,...,n.
$$ 
  We consider the following scheme: for every $n\in  \mathbb N^*$ one constructs $u^n$ as follows : 
\begin{itemize} 
  \item $u^n$ belongs to $C([0,T],L^2)$  with the initial condition
     $$u^n(0)=u_0,$$
\item If $t\in[T_{i}^n,T_{i+1}^n]$ with $i\in2{\mathbb N}$ and $i<n$
 \begin{equation}
 \label{uh-2}
  \left\{ 
 \begin{array}{ll} 
 \partial_t u^n -2\varepsilon \Delta u^n =0,\qquad x\in \mathbb R^2, t>0, \\
 \nabla.u^n=0,
\end{array} \right.    
      \end{equation}
\item if $t\in[T_{i}^n,T_{i+1}^n]$  with $i\in2{\mathbb N}+1$ and $i<n$   
 \begin{equation}
 \label{uh-1}
  \left\{ 
 \begin{array}{ll} 
 \partial_t u^n + 2 u^n \cdot\nabla u^n + 2 \nabla P^n=0,\qquad x\in \mathbb R^2, t>0, \\
  \nabla.u^n=0,
 \end{array} \right.    
      \end{equation}
  \end{itemize}    
      Let us note that $u^n$ exists and it is smooth. In fact, the first step ($i=0$) regularizes the solution and so that $u^n(T_1^n)\in H^\infty(\mathbb R^2)$.
      By the classical result of Kato, the Euler system \eqref{uh-2} has a unique solution on $[T_{1}^n, T_2^n]$ which belongs to  $H^\infty$. And then we iterate the same argument to get a (unique) piecewisely smooth solution $u^n$ on $(0,T]$.

\medskip Let us now give a more convenient form of the different systems (\ref{uh-2}) and (\ref{uh-1}), in terms of the vorticity $\omega^n:= \textrm{curl}(u^n)$.

The system (\ref{uh-2}) can be exactly solved by the heat semigroup (since it preserves the vanishing divergence and commutes with the $\textrm{curl}$-operator), we may rewrite (\ref{uh-2}) as following : for  $t\in[T_{i}^n,T_{i+1}^n]$ with $i\in2{\mathbb N}$ and $i<n$
\begin{equation} 
\omega^n(t)=e^{2\varepsilon(t-T_{i}^n)\Delta}\omega^n(T_{i}^n). \label{wh-2}
 \end{equation}

The system (\ref{uh-1}) may also be written on the vorticity as follows: for $t\in[T_{i}^n,T_{i+1}^n]$ with $i\in2{\mathbb N}+1$ and $i<n$
 \begin{equation}
 \label{wh-1}
  \left\{ 
 \begin{array}{ll} 
 \partial_t \omega^n + 2 u^n \cdot\nabla \omega^n =0,\qquad x\in \mathbb R^2, t>0, 
 \end{array} \right.    
      \end{equation}
      supplemented with the Biot-Savart law:
\begin{equation*}
\label{bs}
u^n=K\ast\omega^n,\quad \hbox{with}\quad K(x)=\frac{x^\perp}{2\pi|x|^2}.
\end{equation*}
As a consquence, we know that $\omega^n(t)=\omega^n(T_i^n) \circ \phi_{t-T_i^n}^{-1}$ where $\phi_{t-T_i^n}$ is the flow corresponding to the vector-field $u^n$.

\medskip

Then, our aim is now to prove that the family $(u^n)_n$ above is uniformly bounded on the interval $[0,T]$, as soon as $T$ is small enough (depending of the initial vorticity). More precisely, one proves the following with the notation
${\mathcal B_{p,\alpha}}:=\LMO \cap L^p$ :

\begin{proposition} \label{prop:uh} Let $p\in[1,2)$ and $\alpha>1$, $\omega_0\in {\mathcal B_{p,\alpha}}$. There exists $T\approx\frac{1}{\|\omega_0\|_{\mathcal B_{p,\alpha}}}$ such that the family   \mbox{$(\omega^n)_{n\in\mathbb N^*}$} is uniformly bounded in $\mathcal B_{p,\alpha}$. More precisely,
$$ 
 \|u^n\|_{L^\infty([0,T],L^2)}+ \|u^n\|_{L^\infty([0,T],\Lip)} + \|\omega^n\|_{L^\infty([0,T],\mathcal B_{p,\alpha})} \leq 2 \| \omega_0\|_{\mathcal B_{p,\alpha}}.
$$

\end{proposition}

\begin{proof}
 For $h\ll T^{-1}$, consider the discrete solution $ \omega^n$ given by (\ref{wh-2}) and (\ref{wh-1}). We write $X_0:=\|\omega(0)\|_{\mathcal B_{p,\alpha}}$ and for $k\in\{1,...,n\}$
 $$ X_k=\sup_{[0,kh]}\|\omega^n(t)\|_{\mathcal B_{p,\alpha}}.
 $$
If $k\in 2{\mathbb N}$ and $k<n-1$ then $\omega^n$ on $[kh,(k+1)h]$ is given
 by (\ref{wh-2}) and so by Remark \ref{remark}
$$ 
X_{k+1} \leq X_k.
$$
If $k\in 2{\mathbb N}+1$ and $k<n-1$ then $\omega^n$ on $[{T}^{n}_k,{T}^{n}_{k+1}]$ is
 given by (\ref{wh-1}) and so from Proposition \ref{prop:velocity}, Proposition \ref{proposition:flowlip} and  Theorem \ref{thm:compo2}, we have 
$$ X_{k+1} \leq X_k \exp(\mu X_{k}h),
$$
for some numerical constant $\mu$ (here we have used that $1+x\leq \exp(x)$ for $x\geq 0$).

As a consequence, the sequence $(X_k)_k$ satisfies the following growth condition: for every $k\in\{1,...,n\}$ 
\begin{equation} 
 \label{eq:xq}
X_k\leq X_{k-1} \exp(\mu{X_{k-1}h})
 \end{equation}
 where $\mu$ is a universal constant.
 
By iteration, we deduce that
\begin{equation} 
X_k\leq X_{0} \exp({\mu(X_0+...X_{k-1})h}).
 \label{eq:xq2} \end{equation}
Let us assume that $X_j\leq 2X_0$ for every $j< k$ then by (\ref{eq:xq2}) we deduce
$$ X_{k} \leq X_0 \exp(\mu T 2X_0)
$$
One chooses $T$ such that
$$
\exp(\mu T 2X_0)=2,
$$
to conclude  
\begin{equation}
X_k \leq 2X_0. \label{eq:co}
\end{equation}
By iterating this reasoning, it comes that (\ref{eq:co}) holds for every $k\leq n$ which combined with Propositions \ref{prop:velocity} and \ref{proposition:flowlip} gives the desired estimate.
\end{proof}

\subsection{Convergence to a solution of Navier-Stokes equation} \label{subsec:conv}

According to Proposition \ref{prop:uh}, there exists a subsequence $(u^n)_{n}$ $*$-weakling converging to $u \in L^\infty([0,T],\Lip)$ and such that $\omega^n$ $*$-weakly converges to $\omega\in L^\infty([0,T],\mathcal B_{p,\alpha})$.

\begin{proposition} The limit $(u,\omega)$ is a solution of 2D Navier-stokes equation
$$ \partial_t \omega + u\cdot \nabla \omega - \varepsilon \Delta \omega = 0$$
and satisfies uniform estimates with respect to $\varepsilon>0$:
$$ \|u\|_{L^\infty([0,T],\Lip)} + \|\omega\|_{L^\infty([0,T],\mathcal B_{p,\alpha})} \leq 2 \| \omega_0\|_{\mathcal B_{p,\alpha}},$$
where $T=T(\omega_0)$ is given in Proposition \ref{prop:uh}.
\end{proposition}

This also proves Theorem \ref{uniform} for the solution of Navier-Stokes equations. We let the reader to check that for the case of fractional Navier-Stokes is exactly the same, since for $\sigma\in(0,1)$ the heat kernel of $e^{-t(-\Delta)^\sigma}$ is given by a non-negative $L^1$-normalized function.

\begin{proof}
 The corresponding estimates on $u$ and $\omega$ directly follows from the uniform estimates of Proposition \ref{prop:uh}. So it remains us to check that $(u,\omega)$ is a solution of 2D Navier-Stokes equation.\\
Let $\phi \in C^\infty([0,T[ \times \R^2)$ compactly supported.
For every $h$ small enough, we have
\begin{equation}
 \sum_{i=0}^{n/2} \int_{T^n_{2i}}^{T^n_{2i+1}} \langle \partial_t \omega^n - 2 \varepsilon \Delta \omega^n , \phi \rangle    ds + \int_{T^n_{2i+1}}^{T^n_{2i+2}}       \langle \partial_t \omega^n + 2 u^n \cdot \nabla \omega^n  , \phi \rangle ds =0 ,
 \label{eq:sum1} 
 \end{equation}
since each term is equal to $0$.  

Using the initial condition on the interval $[  T^n_{2i},T^n_{2i+1}]$ and the vanishing divergence of $u^n$, it comes
$$
 \int_{T^n_{2i+1}}^{T^n_{2i+2}}       \langle \partial_t \omega^n + 2 u^n \cdot \nabla \omega^n  , \phi \rangle ds=
 -\int_{T^n_{2i+1}}^{T^n_{2i+2}} \langle  \omega^n ,  \partial_t \phi + 2 u^n \cdot \nabla \phi \rangle ds + [\langle \omega^n, \phi\rangle ]_{T^n_{2i+1}}^{T^n_{2i+2}} $$
 and
$$
 \int_{T^n_{2i}}^{T^n_{2i+1}} \langle \partial_t \omega^n -2 \varepsilon \Delta \omega^n , \phi \rangle    ds = -\int_{T^n_{2i}}^{T^n_{2i+1}}  \langle  \omega^n ,  \partial_t \phi +2 \varepsilon \Delta \phi \rangle ds + [\langle \omega^n, \phi\rangle ]_{T^n_{2i}}^{T^n_{2i+1}}. $$

So by summing over $i$, (\ref{eq:sum1}) becomes
\begin{equation} \langle \omega_0,\phi(0)\rangle + \int_0^T \langle \omega^n , \partial_t \phi\rangle ds - \sum_{i=0}^{n/2}          
\int_{T^n_{2i+1}}^{T^n_{2i+2}}    \langle  \omega^n , 2 u^n \cdot \nabla \phi \rangle ds +2\varepsilon \int_{T^n_{2i}}^{T^n_{2i+1}}  \langle  \omega^n , \Delta \phi \rangle ds. \label{eq:sum2} \end{equation}

The  family $\partial_t u_n$ is bounded in $L^\infty([0,T], H^{-2})$. Using Ascoli-Arzela and Rellich theorems we get that (up to extract a subsequence) we may assume that the convergence of $u^n$ to $u$ is strong in $L^2([0,T]\times K)$ for every compact $K\subset \mathbb R^2$.

Due to the weak convergence of $\omega_n$, we get
$$ \int_0^T \langle \omega^n , \partial_t \phi\rangle ds \xrightarrow[h\to 0]{}  \int_0^T \langle \omega , \partial_t \phi\rangle ds.$$
We have (using the notation of Lemma \ref{lemma})
$$ \sum_{i=0}^{n/2} \int_{(2i+1)h}^{2(i+1)h} \langle  \omega^n , \Delta \phi \rangle ds = \int_0^T \langle \omega^n(t), (\Delta \phi(t))_h\rangle \, dt.$$
Since $\omega^n=\textrm{curl}(u^n)$, by intgration by parts in the physical space we have
$$ \sum_{i=0}^{n/2} \int_{(2i+1)h}^{2(i+1)h} \langle  \omega^n , \Delta \phi \rangle ds = \int_0^T \langle u^n(t),(\textrm{curl}^* \Delta \phi(t))_h\rangle \, dt.$$
Then using that $(u_n)_n$ strongly converges into $L^2([0,T]\times K)$ (where $K$ is a compact including the space-support of $\phi$) and according to Lemma \ref{lemma} $(\textrm{curl}^* \Delta \phi(t))_h$ weakly converges in $L^2$ then we conclude that
\begin{align*} 
\lim_{h\to 0} \sum_{i=0}^{n/2} \int_{(2i+1)h}^{2(i+1)h} \langle  \omega^n , \Delta \phi \rangle ds & = \frac{1}{2}\int_0^T \langle u(t),\textrm{curl}^* \Delta \phi(t)\rangle \, dt \\
& = \frac{1}{2}\int_0^T \langle \omega(t), \Delta \phi(t)\rangle \, dt.
\end{align*}
For the third term, we decompose $u^n=\tilde u + (u^n-\tilde u)$ with a smooth function $\tilde u$ so that
$$ \langle  \omega^n , 2 u^n \cdot \nabla \phi \rangle = \langle  \omega^n , 2 \tilde u \cdot \nabla \phi \rangle + \langle  \omega^n , 2 (u^n-\tilde u) \cdot \nabla \phi \rangle.$$
As previously, using Lemma \ref{lemma}, we have
$$\sum_{i=0}^{n/2}          \int_{2ih}^{(2i+1)h} \langle  \omega^n , 2 \tilde u \cdot \nabla \phi \rangle ds \xrightarrow[h\to 0]{}   \int_0^T \langle  \omega ,  \tilde u \cdot \nabla \phi \rangle ds.$$
Moreover, 
\begin{align*} 
\left| \sum_{i=0}^{n/2}          \int_{2ih}^{(2i+1)h} \langle  \omega^n, 2 (u^n-\tilde u) \cdot \nabla \phi \rangle ds \right| & \lesssim  \sum_{i=0}^{n/2}          \int_{2ih}^{(2i+1)h} \|\omega^n\|_{L^2} \|u^n- \tilde u\|_{L^2} \| \nabla \phi\|_{L^\infty} ds \\
 & \lesssim T \|\omega^n\|_{L^2([0,T],L^2)} \|\tilde u-u^n\|_{L^2([0,T],L^2({\rm Supp}(\phi))} \\
 & \lesssim  \|\tilde u-u^n\|_{L^2([0,T],L^2({\rm Supp}(\phi))}.
\end{align*}
 So finally, using the local strong convergence in $L^2([0,T],L^2)$ of $(u^n)_n$ we have
 $$ \limsup_{h\to 0} \left| \sum_{i=0}^{n/2}          \int_{2ih}^{(2i+1)h} \langle  \omega^n , 2 u^n \cdot \nabla \phi \rangle \,ds - \int_0^T \langle  \omega , 2 u \cdot \nabla \phi \rangle \right| \lesssim \inf_{\tilde u \in C^\infty_0} \|\tilde u - u \|_{L^2([0,T],L^2)} = 0.$$
 
So taking the limit when $h\to 0$ in (\ref{eq:sum2}) yields
\begin{equation} \langle \omega_0,\phi(0)\rangle + \int_0^T \langle \omega , \partial_t \phi\rangle ds - \int_0^T \langle  \omega ,  u \cdot \nabla \phi \rangle ds -\varepsilon \int_{0}^{T} \langle  \omega , \Delta \phi \rangle ds, \end{equation}
which by integrations by parts in time gives (in a distributional sense)
\begin{equation} \int_0^T \langle \partial_t \omega + u \cdot \nabla \omega - \varepsilon\Delta \omega , \phi \rangle ds =0 \end{equation}
This last equality holds for every compactly supported smooth functions $\phi \in C^\infty_0([0,T)\times \R^2)$, so we deduce that $(\omega,u)$ is a solution of Navier-Stokes equation. By uniqueness of solution, $(\omega,u)$ is the solution of Navier-Stokes equation and satisfies the uniform estimates.
\end{proof}

\begin{lemma} \label{lemma} With the previous notations. Let $f$ be a compactly supported smooth function on $[0,T] \times \R^2$, then
$$ f_h:= \sum_{i=0}^{N/2} {\bf 1}_{[(2i+1)h,2(i+1)h]} (t) f $$
weakly converges in $L^2([0,T],L^2)$ to $\frac{1}{2} f$ when $h$ goes to $0$.
\end{lemma}

\begin{proof} Since $(f_h)_{h>0}$ is uniformly bounded in $L^2([0,T],L^2)$, it suffices us to check that for every smooth function $g$
\begin{equation}
\lim_{h\to 0 }  \iint f_h(t,x) g(t,x) \, dtdx = \frac{1}{2} \iint f(t,x) g(t,x) \, dtdx. \label{am}
\end{equation}
So let us fix such a function $g$ and set
$$ \widetilde{f}_h:= \sum_{i=0}^{N/2} {\bf 1}_{[2ih,(2i+1)h]} (t) f $$
such that $f= f_h+ \widetilde{f}_h$.
However, it is clear (by a first order expansion in the time variable) that 
$$ \left|\int_{2ih}^{(2i+1)h}\int f(t,x) g(t,x) dxdt -  \int_{(2i+1)h}^{2(i+1)h}\int f(t,x) g(t,x) dxdt \right| \leq \| \partial_t(fg)\|_{L^\infty} h^2. $$
So summing over $i$ yields
$$ \left| \iint f_h(t,x) g(t,x) \, dtdx - \iint \widetilde{f}_h(t,x) g(t,x) \, dtdx\right| \lesssim h $$
and so 
$$ \lim_{h\to 0} \left| \iint f_h(t,x) g(t,x) \, dtdx - \iint \widetilde{f}_h(t,x) g(t,x) \, dtdx\right|  = 0.$$
We conclude with the equality $ f= f_h+ \widetilde{f}_h$ which gives
$$ \frac{1}{2}f-f_h = \frac{1}{2}\left(f_h-\widetilde{f}_h\right).$$
\end{proof}

\end{document}